\documentclass[12pt]{article}
\usepackage[numbers,square]{natbib}
\usepackage{mathrsfs}
\usepackage{amssymb}
\usepackage{amsthm}               
\usepackage{amsmath}
\usepackage{tikz}

\newtheorem{thm}{Theorem}[section]
\newtheorem{lem}[thm]{Lemma}
\newtheorem{cor}[thm]{Corollary}
\newtheorem{question}[thm]{Question}

\newcommand{\hs}{Sprig}
\newcommand{\mpl}{mis\`{e}re-play}
\newcommand{\np}{normal-play}

\newcommand{\R}{\mathcal{R}}
\renewcommand{\L}{\mathcal{L}}
\newcommand{\N}{\mathcal{N}}
\renewcommand{\P}{\mathcal{P}}

\newcommand{\cgc}{\!:\!}
\newcommand{\Gb}{\overline{G}}
\newcommand{\Hb}{\overline{H}}

\newcommand{\combgame}[1]{#1}
\newcommand{\cgstar}{\ast}
\newcommand{\cgfuzzy}{\not\gtrless}

\colorlet{colour0}{green}
\colorlet{colour1}{blue}
\colorlet{colour2}{red}
\newcommand{\hbstrings}[1]{
  \path[-] (0,0) edge (0,1);
  \foreach \edgeset [count=\strnum] in #1 {
    \path[-] (0,\strnum) edge (0,\strnum+1);
    \foreach \e [count=\enum] in \edgeset {
      \path[-,colour\e] (\enum-1,\strnum) edge (\enum,\strnum);
      \node at (\enum-1,\strnum){};
    }
    \node at (\enum,\strnum){};
  }
}

\begin{document}

\title{Mis\`{e}re-play Hackenbush Sprigs}

\author{Neil A. McKay, Rebecca Milley, Richard J. Nowakowski\\
Department of Mathematics and Statistics\\
Dalhousie University\\
Halifax, NS, Canada}
\maketitle

\begin{abstract}
A Hackenbush Sprig is a Hackenbush String with the ground edge colored green and the remaining edges either red or blue. We show that in canonical form a Sprig is a star-based number (the ordinal sum of star and a dyadic rational) in mis\`{e}re-play, as well as in normal-play. We find the outcome of a disjunctive sum of Sprigs in mis\`{e}re-play and show that it is the same as the outcome of that sum plus star in normal-play. Along the way it is shown that the sum of a Sprig and its negative is equivalent to 0 in the universe of mis\`{e}re-play dicotic games, answering a question of Allen.
\end{abstract}


\section{Introduction}
The game \textsc{hackenbush} \citep{BerleCG2001,Conwa2001}
 is played on a finite graph with edges colored blue, red, and green. 
There is one special vertex called the ground, shown by a long horizontal line.
The player Left can move by cutting a blue or green edge (Right a red or green edge)
and removing any portion of the graph no longer connected to the ground. Under \textit{normal-play}, the first player who cannot move loses; under \textit{mis\`ere-play}, the first player who cannot move wins. 
 In \citep{BerleCG2001,Conwa2001} many of the concepts of combinatorial game theory are exhibited using \np\ hackenbush. However, the game is still not completely understood. For example, both normal- and \mpl\ \textsc{hackenbush} are NP-hard \citep{BerleCG2001,Conwa2001,Stewa}.

 A \textsc{hackenbush string} (String) is a path with one end rooted in the ground.
A \textsc{hackenbush flower} (Flower) is a String composed of green edges with loops on top, all red or all blue.
  A \textsc{hackenbush sprig} (\hs) is a String with only the rooted edge being green; Figure \ref{fig:hbsprigs} shows a sum of \hs s.  

  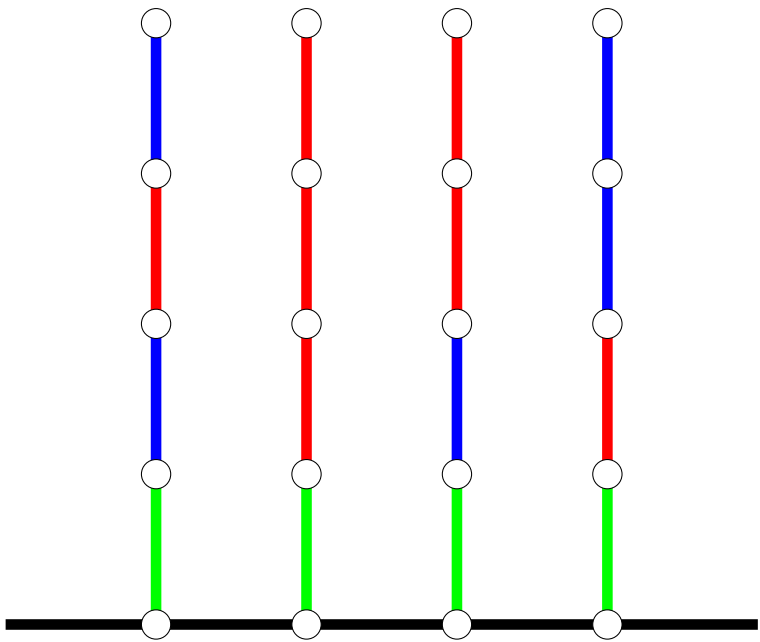
\begin{figure}[ht]
    \centering
    \begin{tikzpicture}[
      rotate=90,
      scale=2,
      every node/.style={shape=circle, draw, fill=white},
      every edge/.style={line width=4, draw}
      ]
      \hbstrings{{{0,2,1,1},{0,1,2,2},{0,2,2,2},{0,1,2,1}}}
    \end{tikzpicture}
    \caption{A game of Hackenbush Sprigs}
    \label{fig:hbsprigs}
  \end{figure}
  
Strings, Sprigs, and Flowers are ordinal sums.
The \textit{ordinal sum} of $G$ and $H$ is the game $G\cgc H=\combgame{\{{\mathcal{G}}^L,G\cgc {\mathcal{H}}^L| {\mathcal{G}}^R,G\cgc {\mathcal{H}}^R\}}$.
We refer to $G$ as the \textit{base} and $H$ as the \textit{dependent}.
In an ordinal sum, a move in the base prevents further play in the dependent, but a move in the dependent does not affect play in the base.
A position of the form $\cgstar \cgc H$ for some $H$ is called \textit{star-based}. Both a Flower and a \hs\ are star-based positions.

A String of $n$ green edges has the same game tree as a nim-heap $\cgstar_n =\combgame{\{0,\cgstar, \cgstar_2,\ldots,\cgstar_{n-1} | 0,\cgstar, \cgstar_2,\ldots,\cgstar_{n-1} \}}$ of size $n$, and so $\cgstar_n=\cgstar\cgc\cgstar_{n-1}$. A Flower is either $\cgstar_n\cgc k$  or $\cgstar_n\cgc -k$ for some positive integer $k$.
A long-standing question of Berlekamp's regarding Flowers (see \citep{BerleCG2001}) can then be rephrased as a question about star-based games:

\medskip
\emph{Who wins $\sum_{i=1}^m\cgstar_{n_i}\cgc k_i$ where $\{n_1,n_2,\ldots,n_m\}$ is a set of positive integers and $\{k_1,k_2,\ldots,k_m\}$ is a set of integers?} 
\medskip

 In this paper we concentrate on \hs s. As hinted at in Berlekamp's Flower question, a disjunctive sum of star-based positions is most often not a star-based position.
In Section \ref{sec:miseresums} we answer our main questions. Namely, if $\{x_1,x_2,\ldots,x_m\}$ is a multi-set of numbers:
  \begin{itemize}
  \item[] \emph{In \mpl, who wins  $\sum_{i=1}^m\cgstar\cgc x_i$?} (Theorem \ref{thm:outcomenostar});
\item[] \emph{In \mpl, who wins $\cgstar+\sum_{i=1}^m\cgstar\cgc x_i$?} (Theorem \ref{thm:outcomewithstar}).
  \end{itemize}  

\emph{In \np, who wins $\sum_{i=1}^m\cgstar\cgc x_i$ and  $\cgstar+\sum_{i=1}^m\cgstar\cgc x_i$?} was essentially analyzed by Conway, \citep{Conwa2001}, and is presented in Section \ref{sec:normalplay} using the concepts and algorithms developed in this paper.
 
 In the next sub-section, we give the definitions and concepts pertinent to this paper; for further definitions and background see \citep{AlberNW2007,BerleCG2001,Conwa2001}. In Section \ref{sec:miseregeneral}, we present the general results needed for analyzing \mpl\ \hs s, in particular proving that $\cgstar\cgc x -\cgstar\cgc x\equiv 0$ in some universes (Corollary \ref{cor:starcolonconjugate}). This is an important result since, in the universe of all \mpl\ games, $G-G\not\equiv 0$ unless $G = \combgame{\{\cdot | \cdot\}}$. In \np\ $G-G=0$, is proved by following the \textit{Tweedledum-Tweedledee}, or copycat, strategy: mirror your opponent's move but in the opposite summand. This is a bad strategy in \mpl!
 
 \subsection{Background}\label{sec:background}

Under \np\ and \mpl, all games have a unique canonical form obtained by eliminating dominated options and by-passing reversible options.
 In \np, many canonical positions have acquired names. In \mpl, positions with those same game trees also appear frequently and we refer to them by their \np\ names. For example, $0=\combgame{\{\cdot|\cdot\}}$; $\cgstar=\combgame{\{0|0\}}$; for a positive integer $n$, $n=\combgame{\{n-1|\cdot\}}$; and for integers $p,q$, $\frac{2p+1}{2^q}=\combgame{\{ \frac{p}{2^{q-1} }|\frac{p+1}{2^{q-1} }\}}$ provided $\gcd(p,q)=1$. 

Let $\N$, $\P$, $\R$, and $\L$ represent the outcomes classes of \textit{Next-, Previous-, Right-,} and \textit{Left-}win games, respectively. In both play conventions, these are partially ordered as in Figure \ref{fig:outcomeclasses}.

\begin{figure}[ht]
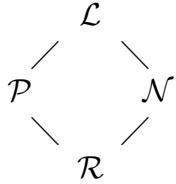

  \centering
  \begin{tabular}{c}
$\L$\\
$\diagup \quad \quad \diagdown$\\
$\P \quad \quad \; \;	\quad \N$\\
$\diagdown \quad \quad \diagup$\\
$\R$\\
\end{tabular}
  \caption{The partial order of outcome classes}
  \label{fig:outcomeclasses}
\end{figure}

 The outcome functions $o^+(G)$ and $o^-(G)$ give the outcome class of $G$ under normal- and \mpl\ respectively. For example, $o^+(0) = \P$ and $o^-(0) = \N$. Similarly, we write $0 \in \P^+$ and $0 \in \N^-$.

For this paper, there are important concepts that have to be handled with care. The \textit{negative} of a game is denoted $-G=\combgame{\{-\mathcal{G}^R|-\mathcal{G}^L\}}$. In \np, the negative is the additive inverse; that is, $G+ (-G)=0$. In \mpl, whether the negative is also the additive inverse depends on the universe being considered. To avoid confusion and inappropriate cancellation, in \mpl\ we represent the negative by $\overline{G}$ instead of $-G$.
In \np\ and \mpl\ \textit{equality} and \textit{inequality} are defined as follows:
 \begin{itemize}
\item[] $G=^{\between} H$ if $o^{\between}(G+X)= o^{\between}(H+X)$ for all games $X$;
\item[] $G\geq^{\between} H$ if $o^{\between}(G+X)\geq o^{\between}(H+X)$ for all games $X$.
\end{itemize}
\noindent where ${\between}$ is either $+$, denoting \np, or $-$, denoting \mpl.
In \np\ games, there are easy tests for equality and inequality:
\begin{itemize}
\item[] $G=0$ iff $G\in \P$;
\item[] $G=H$ iff $G-H=0$;
\item[] $G>H$ iff $G-H>0$.
\end{itemize} 
 
 In \mpl, no such tests exist and equality involves comparing with all games.
However, Plambeck (see \citep{PlambS2008}, for example) had a breakthrough in dealing with impartial mis\`ere games: restrict the universe to which $X$ may belong. Allen \citep{Allen2009}, when exploring partizan mis\`ere games, used the same idea.

Given a universe $\mathcal{X}$,
\begin{itemize}
\item[] $G\equiv H\pmod{\mathcal{X}}$
 if $o^-(G+X)= o^-(H+X)$ for all games $X\in\mathcal{X} $;
\item[] $G\geq^-H\pmod{\mathcal{X}}$
 if $o^-(G+X)\geq o^-(H+X)$ for all games $X\in\mathcal{X} $.
\end{itemize} 

If $G \equiv H$ then $G$ and $H$ are said to be \emph{indistinguishable}. Otherwise, there exists a game $X$ such that $o^-(G+X) \not = o^-(H+X)$ and we say $G$ and $H$ are \emph{distinguishable}.

To avoid the proliferation of superscripts in this paper, we will distinguish between \np\ and \mpl\  relations by reserving $=, >, <,\geq, \leq$ for the order relations in \np\ and $\equiv,\gtrdot,\lessdot,\geqq,\leqq$ for the corresponding relations in \mpl, which should also be accompanied by a reference to a universe.
 The universes of Allen, \citep{Allen2011}, and Plambeck and Siegel, \citep{PlambS2008}, are defined by starting with a single game $G$ and adding all the followers and all disjunctive sums that can be formed.
 Distinguishable elements in a given universe will also be distintinguishable in a larger universe; we make implicit use of this throughout the paper and so present it as a lemma below.
\begin{lem}
Let $\mathcal{X}$and $\mathcal{Y}$ be sets of games with $\mathcal{X}\subset \mathcal{Y}$. 
If $G\not\equiv H\pmod{\mathcal{X}}$ then $G\not\equiv H\pmod{\mathcal{Y}}$.
\end{lem}
 
 \begin{proof}
 If there is some game $X\in\mathcal{X}$ with $o^-(G+X)\neq o^-(H+X)$  
 for $X\in \mathcal{X}$, then $X$ also 
 distinguishes $G$ and $H$ in the universe $\mathcal{Y}$.
 \end{proof}

  Any position in \np\ with the property that in all positions either both players have a move or neither does, is an \textit{infinitesimal} game, so it is reasonable to call this class of positions \textit{all-small} (see \citep{Conwa2001}, page 101). In \mpl\ these games are not infinitesimal and thus the term all-small is inappropriate.  We\footnote{`We' means Meghan Allen, Alan Guo, Neil McKay, Erza Miller, Richard Nowakowski, Thane Plambeck, Aaron Siegel, Angela Siegel, and Mike Weimerskirsch.}, carrying the gardening theme, introduced the term \textit{dicot} for a position in which all positions, have options for both players or neither player. For example, a \hs\ is a dicot, as are sums of \hs s.
Let $\mathcal{D}$ denote the class of all dicot games and $\mathcal{S}$ the set of all positions that are a finite sum of \hs s. As $\mathcal{S}\subset \mathcal{D}$, if two elements of $\mathcal{S}$ are distinguishable in $\mathcal{S}$ then they are distinguishable in $\mathcal{D}$. 

As mentioned earlier, in the entire mis\`ere universe, $G+\overline{G}\not \equiv 0$ unless $G=\combgame{\{\cdot|\cdot\}}=0$. In \citep{Allen2009}, it is shown that $\cgstar + \cgstar \equiv 0\pmod{\mathcal{D}}$; Allen further asks: \textit{For which $G$ is $G+\overline{G}\equiv 0\pmod{\mathcal{D}}$?} We shall see that every \hs\ has this desirable property.

In \np, a red-blue String is a number \citep[page 77]{BerleCG2001}, so a \hs\ is equal to $\cgstar\cgc x$ for some number $x$. Moreover, by \citep[Theorem 94]{Conwa2001}, $\cgstar\cgc x$ is in canonical form.
Culminating in Theorem \ref{thm:numbersarenumbers}, we show below that this is also true in \mpl\ restricted to the dicot universe.

\section{Mis\`ere-play Sprigs}\label{sec:miseregeneral}
In this section we present general results useful for analyzing \hs s in \mpl.
\begin{thm}\label{thm:starcolon}
  For any game $G$, $o^-(\cgstar\cgc G) = o^+(G)$.
\end{thm}
\begin{proof}
  The game $\cgstar\cgc G$ will not end until one of the players moves in $\cgstar$, and at that point the player doing so loses. The only way to guarantee a win in the ordinal sum is to make the last move in $G$.  
\end{proof}

\begin{cor}
If $x$ is a number then
 \[
o^-(\cgstar\cgc x)=
\begin{cases}
L,\mbox{ if } x>0;\\
P,\mbox{ if } x=0;\\
R,\mbox{ if } x<0.
\end{cases}
\]
\end{cor}

\begin{thm} \label{thm:G-G} If $G+\Gb \in \N^-$ and $H+\Hb \in \N^-$ for all followers $H$ of $G$, then $G+\Gb \equiv 0 \pmod{\mathcal{D}}$.
\end{thm}

\begin{proof}
Let $X$ be an finite sum of games in $\mathcal{D}$ and suppose Left wins $X$. We give a strategy for Left to win $G+\Gb+X$. Left follows her original strategy for $X$ unless no move is available in $X$ for Left (or Right) in which case Left plays her winning move in $G+\Gb$. If Right at some point plays in $G+\Gb$, Left mirrors Right's move leaving $G^R+\overline{G^R}$ or $G^L+\overline{G^L}$, which is equivalent to $0$ by induction. Right must resume play in $X$ and thus loses.
\end{proof}

\begin{cor} \label{cor:star+star}{\rm\citep{Allen2011}}
In the dicot universe, $*+* \equiv 0\pmod{\mathcal{D}}$.
\end{cor}

\begin{cor} \label{cor:starcolonconjugate}
  If $x$ is a number then  $\cgstar\cgc x + \cgstar\cgc \overline{x} \equiv 0\pmod{\mathcal{D}}$.
\end{cor}
\begin{proof}
If $x=0$ then this is just Corollary \ref{cor:star+star}.  
Otherwise, assume without loss of generality that $x>0$ (i.e., $x\in \L^+$). As $\overline{\cgstar\cgc x}$ is the same position as $\cgstar\cgc \overline{x}$, and any non-zero follower of $\cgstar\cgc x$ is $\cgstar\cgc x'$ for a number $x'$, it suffices by Theorem \ref{thm:G-G} to show $\cgstar\cgc x+\cgstar\cgc \overline{x} \in \N^-$ for any number $x$. Left playing first on $\cgstar\cgc x+\cgstar\cgc \overline{x}$ moves  $\cgstar\cgc \overline{x}$ to $0$ and then wins playing second under \mpl\ on $\cgstar\cgc x$, by Theorem \ref{thm:starcolon}. Since $\overline{x} \in \R^+$, Right can similarly win this sum playing first, and so $\cgstar\cgc x+\cgstar\cgc \overline{x} \in \N^-$. 
\end{proof}

\begin{thm}\label{thm:miserecolon1}
 Let $G$ have a left and a right option. If $H \geq 0$ then in any universe $\mathcal{X}$, $G \cgc H \geqq G \pmod{\mathcal{X}}$.
\end{thm}
\begin{proof}
Given a strategy in $G + X$, where $X \in \mathcal{X}$, Left can do at least as well in $G \cgc  H + X$ by following this amended strategy: if Right plays in $H$ then respond in $H$, otherwise follow the original strategy for $G + X$. As $H \ge 0$, Left can always respond to Right's moves in $H$. Also, $G$ has options for both players so the addition of $H$ is of no benefit to Right.
\end{proof}

\begin{cor}\label{cor:miserecolon2}
Let $x$ and $y$ be red-blue Strings. If $y > 0$ then
\[\cgstar\cgc x\cgc y \gtrdot \cgstar\cgc x\pmod{\mathcal{S}}.\]
\end{cor} 
\begin{proof}
 By Theorem \ref{thm:miserecolon1}, $\cgstar\cgc x\cgc y  \geqq \cgstar\cgc x$. It remains to find a game in $\mathcal{S}$ that can distinguish $\cgstar\cgc x\cgc y$ and $\cgstar\cgc x$. Consider $\cgstar\cgc \overline{x} + \cgstar \in \mathcal{S}$. The sum $\cgstar \cgc x + \cgstar \cgc \overline{x} + \cgstar$ is in $\P^-$ as $\cgstar\cgc x+\cgstar\cgc \overline{x} + \cgstar \equiv \cgstar \pmod{\mathcal{D}}$ by Corollary \ref{cor:starcolonconjugate}.
Also, because $y>0$, it starts with a blue edge. Going first in $\cgstar\cgc x\cgc y + \cgstar\cgc \overline{x} + \cgstar$ Left wins by moving to $\cgstar\cgc x+\cgstar\cgc \overline{x} + \cgstar$ and thus we have distinguished the game in question. 
\end{proof}

\begin{thm}\label{thm:orderedsprigs}
Let $x$ and $y$ be red-blue Strings. If $x > y$ then
\[\cgstar\cgc x \gtrdot \cgstar\cgc y\pmod{\mathcal{S}}.\]
\end{thm} 

\begin{proof}
Let $z$ be the common part of $x$ and $y$ such that $x = z \cgc x'$ and $y = z \cgc y'$. At most one of $x'$ and $y'$ is empty, and $x'$ may only start with a blue edge and $y'$ may only start with a red edge. By Corollary \ref{cor:miserecolon2}, $\cgstar\cgc x \equiv \cgstar\cgc z \cgc x' \geqq \cgstar \cgc z \geqq \cgstar\cgc z \cgc y' \equiv \cgstar\cgc y\pmod{\mathcal{S}}$ and equivalence does not hold in throughout because $x \not = y$.
\end{proof}

In \np, a \hs\ with $x$ as the value of the red-blue String is easily seen to have the value $\cgstar\cgc x$ since all moves are dominated except for $\combgame{\{0,\cgstar\cgc x^L|0,\cgstar\cgc x^R\}}$. We now show the same for \mpl.

\begin{thm}\label{thm:numbersarenumbers}
  In \mpl\ the canonical form of a \hs\ is $\cgstar \cgc x$ where $x$ is the \np\ value of the red-blue part of the \hs.
\end{thm}
\begin{proof}
In the red-blue String corresponding to $x$, removing the top blue (red) edge is the best move in \np\ for Left (Right); any other move is strictly dominated. By  Theorem \ref{thm:orderedsprigs}, the same holds true in \mpl\ Sprigs for red or blue moves. Removing the green edge leaves 0 which is incomparable with $\cgstar \cgc g$ if $g\not = 0$ because $\cgstar + 0\in \P^-$ and $\cgstar+\cgstar \cgc g\in \N^-$.
For reversibility, we assume without loss of generality that Left is moving first. If Right chops a red edge in response, the result is not less than the original position; if Right responds by chopping the green edge, we are at $0$ which is incomparable with the original.
Therefore the canonical form of the \hs\ is $\combgame{\{0,\cgstar\cgc x^L|0,\cgstar\cgc x^R\}}=\cgstar\cgc x$.
\end{proof}

\section{Outcomes of Sums}\label{sec:miseresums}

In this section, we give the \mpl\ outcomes for disjunctive sums of \hs s positions. 
If $X$ and $Y$ are multisets of positive numbers then let $(X,Y)=\sum_{x\in X}\cgstar\cgc x + \sum_{y\in Y}\cgstar\cgc -y$.
 We call $G=(X,Y)$ \emph{reduced} if $X \cap Y = \emptyset$.
Note that for any game $H$, $\overline{\cgstar\cgc H}$ is the same game as $\cgstar\cgc \Hb$.
 In \np\ $\cgstar\cgc H +\cgstar\cgc \Hb = 0$; Corollary \ref{cor:starcolonconjugate} shows this is also true in \mpl\  in the dicot universe.
Thus, given a position $G = (X,Y)$, there is a \emph{unique reduced} position $G' = (X',Y')$, where $X' = X \setminus Y$ and $Y' = Y \setminus X$.
 
 We define the \textit{advantage} of $G=(X,Y)$ to be $\Delta(G)=|X|-|Y|$ (which is the same as $|X'|-|Y'|$). We define the \textit{edge} of $G$ to be $\epsilon(G)=\min(X') - \min(Y')$. If $X'$ or $Y'$ is empty then we take $ \epsilon(G)=0$.

\begin{lem}\label{zerozero}
If $G = (X,Y)$ with $\Delta(G) = 0$ and $\epsilon(G) = 0$, then $X=Y$ and $G\equiv 0\pmod{\mathcal{D}}$.
\end{lem}

\begin{proof}
As $\Delta(G) = 0$, $|X| = |Y|$ and $|X'| = |Y'|$. As $\epsilon(G)=0$, at least one of $X'$ or $Y'$ must be empty, so they both must be as they are the same size.
\end{proof}

A \hs s position is of the form $(X,Y)$ or $(X,Y) + \cgstar$. The advantage and edge of a position are sufficient to determine the outcome of the position.

\begin{thm}\label{thm:outcomenostar}
If $G=(X,Y)$ then 
\[ o^-(G)=
\begin{cases}
\L \mbox{ if } \Delta(G)>0;\\
\R \mbox{ if } \Delta(G)<0;\\
\N \mbox{ if } \Delta(G)=0.
\end{cases}
\]
\end{thm}

\begin{proof}
Let $G=(X,Y)$. We proceed by induction on $|X|+|Y|$.
If $|X|+|Y|=0$ then $G=0$ which is a next-player win. If $|X|=1$ and $|Y|=0$ then $G\in\L^-$ by Theorem \ref{thm:starcolon}.

Suppose $|X|=|Y|>0$.
As $|Y| > 0$, Left going first can move some $\cgstar\cgc \overline{y}$ to $0$; the resulting position, $G^L$, has $\Delta(G^L) >0$ and so Left wins by induction.
Similarly, can win Right moving first. Thus, if $\Delta(G) = 0$, then $G \in \N^-$.

Suppose $|X|>|Y|$. If $|Y| > 0$, then Left wins going first as above.
If $|Y|=0$ then we need only consider $|X|>1$. Left wins going first by moving $\cgstar \cgc x$ to $0$ for some $x \in X$, which is a winning move by induction. If Right moves first to $G^R$ such that $\Delta(G^R) \ge 0$, then Left wins because $G^R \in \N^-\cup \L^-$; if Right's move does not change $\Delta$, then Left wins using the first-player argument previously given.

Omitted arguments for $|X| < |Y|$ are similar.
\end{proof}

\begin{lem}\label{lem:withstargoingfirst}
  Let $G = (X,Y)$ and consider $G+\cgstar$. If $\Delta(G)>0$ then
  \begin{itemize}
  \item Left can win by playing first;
  \item if Right can win playing first, he can do so moving $\cgstar \cgc x$ to $0$ where $x = \max(X')$.
  \end{itemize}
\end{lem}
\begin{proof}
  Playing in $G + \cgstar$, if $\Delta(G)>0$ then Left can win by playing first by moving $\cgstar$ to $0$, leaving $G$ which is winning by Theorem \ref{thm:outcomenostar}.

If Right does not move a \hs\ to 0, then he must play in some $\cgstar \cgc x$ to $\cgstar \cgc x^R$. As $x^R > x$, Theorem \ref{thm:orderedsprigs} then says that $\cgstar \cgc x^R \gtrdot \cgstar \cgc x$. By the definition of $\gtrdot$, the new position is better for Left and so Left wins playing first from this position.

Right can only win by eliminating a \hs. By Theorem \ref{thm:orderedsprigs} Right's options, such as $G + \cgstar - \cgstar \cgc x_1$, are ordered and so Right eliminates the \hs\ that is best for the opponent.
\end{proof}

\begin{thm}\label{thm:outcomewithstar}
If $G=(X,Y)$ then \[ o^-(G +\cgstar) =
\begin{cases}
\L \mbox{ if } \Delta(G)>1 \mbox{ or }  \Delta(G)=0,1 \mbox{ and } \epsilon(G)>0;\\
\R \mbox{ if } \Delta(G)<-1 \mbox{ or }  \Delta(G)=0,-1\mbox{ and } \epsilon(G)<0;\\
\N \mbox{ if } \Delta(G)=1  \mbox{ and } \epsilon(G)\leq 0 \mbox{ or } \Delta(G)=-1 \mbox{ and } \epsilon(G)\geq 0;\\
\P \mbox{ if } \Delta(G)=0  \mbox{ and } \epsilon(G)=0.\\
\end{cases}
\]
\end{thm}

\begin{proof}
The outcome is the same as the outcome of the reduced game, so assume $G=(X,Y)$ is reduced and let $H = G + \cgstar$. We proceed by induction on $|X|+|Y|$. We focus on Left; omitted arguments for Right are similar.
\begin{itemize}
\item If $|X|+|Y|=0$ then $H = \cgstar$ which is a previous-player win.
\item If $|X|+|Y|=1$ then $H\in\N^-$ as either player wins by moving to $\cgstar$.
\item If $|X|=|Y|=1$, then $H = \cgstar \cgc x + \cgstar \cgc \overline{y} +\cgstar$. The first player to move any summand to $0$ loses; play proceeds in $x$ and $\overline{y}$ until someone is forced to do so, at which point the opponent responds by eliminating a second \hs, leaving a \hs\ that is at least as good for them as $\cgstar$. In particular, $o^-(H) = o^+(x-y) = o^+(\epsilon(G))$. 
\item If $|X|=|Y|>1$ and $\epsilon(G) > 0$, then Left wins going first in $H$ by moving one of Right's \hs s to $0$, changing her advantage to $1$. By Lemma \ref{lem:withstargoingfirst}, $H^L$ is winning for Left moving first; Right must respond by moving one of Right's \hs s to $0$, leaving $H^{LR}$. This position is in $\L^-$ by induction as both Left and Right will have removed the opponent's best Sprig and the edge has not changed.
 As $H$ is a win for Left going first, Right's must respond moving one of Left's \hs s to $0$, to which Left responds to $G^{RL} = G^{LR}$, a winning move.
\item If $|X|>|Y|$ then Left can win going first by Lemma \ref{lem:withstargoingfirst}. Right's best move going first is to move a sprig to $0$. If $|X|-|Y|>2$, Right's move loses. However, if $|X|-|Y|=1$ and $\epsilon(G) <0$, then Right wins by induction.
\end{itemize}
\end{proof}
\section{Distinguishability and  Monoid}

In \np, any given game $G$ is equivalent to many other games. It is generally true that in \mpl\ there are few games equivalent to (indistinguishable from) $G$. In impartial \mpl, `moding out' by the indistinguishablity relations leads to a monoid that replaces the nimbers of \np. This was part of the Plambeck's breakthrough (\citep{PlambS2008}). 
It is of less use here but we present the monoid for completeness.

\begin{lem} 
Let $G=(X_1,Y_1)$ and $H =(X_2,Y_2)$ be reduced games. If $\delta\in \{0,1\}$, then $G + \delta \cdot \cgstar \equiv H + \delta\cdot \cgstar\pmod{\mathcal{D}}$  if and only if $X_1=X_2$ and $Y_1=Y_2$.
\end{lem}

\begin{proof}
The sufficiency is clear. Conversely, suppose without loss of generality that 
$X_1 \neq X_2$. In this case, the reduced form of $G+\overline{H}$ is not $0$. The games $G+\cgstar$ and $H+\cgstar$ are now distinguished by $\Hb$, since by Lemma \ref{thm:outcomewithstar}, $G+\overline{H}+\cgstar \not \in \P^-$ while $H+\overline{H}+\cgstar \equiv \cgstar\pmod{\mathcal{D}}$ and $\cgstar \in \P^-$.  Similarly $G$ and $H$ are distinguished by $\overline{H}+\cgstar$.
\end{proof}

\begin{lem}
If $G=(X_1,Y_1)$ and $H =(X_2,Y_2)$, then $G \not \equiv H +\cgstar \pmod{\mathcal{S}}$.
\end{lem}

\begin{proof}
The games are distinguished by $\overline{H}$, since $H+\overline{H}+\cgstar\in \P^-$ while 
$G+\overline{H}\not \in \P^-$.
\end{proof}

\begin{cor}
Let $G=(X_1,Y_1)$ and $H =(X_2,Y_2)$, $G \equiv H \pmod{\mathcal{D}}$ if and only if $X_1=X_2$ and $Y_1=Y_2$.
\end{cor}

Our only reductions are $\cgstar+\cgstar\equiv 0\pmod{\mathcal{D}}$ and $\cgstar:x+\cgstar:-x\equiv 0\pmod{\mathcal{D}}$.  We can now describe the mis\`{e}re monoid associated with Sprigs, $\mathscr{M}_{\mathcal{S}}$, which we write multiplicatively by convention.

For each positive number $\frac{p}{2^q}$, with $p\equiv 1 \pmod{2}$, map $\cgstar:\frac{p}{2^q}$ to $\kappa_{p,q}$ and $\cgstar:\frac{-p}{2^q}$ to $\kappa_{p,q}^{-1}$, and let $K$ be the set of all resulting variables.  These mappings along with $0\mapsto e$ and $\cgstar \mapsto \alpha$ give us the monoid
\[\mathscr{M}_{\mathcal{S}} = \langle e,\alpha,K \; \vert \; \alpha^2=e, \kappa_{p,q}\cdot \kappa_{p,q}^{-1}=1 \; \forall \kappa_{p,q}\in K \rangle.\] The outcome tetrapartition, i.e.~the sets $\mathcal{S}\cap\L^-$, $\mathcal{S}\cap\R^-$, $\mathcal{S}\cap\N^-$, and $\mathcal{S}\cap\P^-$, has only one finite member, namely $\mathcal{S}\cap\P^-=\{\cgstar\}$. The others sets are given implicitly by Theorems \ref{thm:outcomenostar} and \ref{thm:outcomewithstar}.

\section{Normal-play Sprigs}\label{sec:normalplay}

Using different notation (and without \textsc{hackenbush}) Conway \citep[page 194]{Conwa2001} considered the Sprigs problem in \np; he gives a characterization of the outcomes of sums of such positions. We present an alternative to his Theorem 88 in our terminology. 

\begin{thm}\label{thm:normal}
If $G=(X,Y)=\sum_{x\in X}\cgstar\cgc x + \sum_{y\in Y}\cgstar\cgc -y$ then
\begin{itemize}
\item[(i)] \[ G 
\begin{cases}
>0  &\mbox{if } \Delta(G)>1  \mbox{ or } \Delta(G)=0,1 \mbox{ and } \epsilon(G)>0;\\
 = 0 &\mbox{if } \Delta(G)=0 \mbox{ and } \epsilon(G)=0;\\
 \cgfuzzy 0 &\mbox{if } \Delta(G)=1 \mbox{ and } \epsilon(G)\leq 0
 \mbox{ or } \Delta(G)=-1 \mbox{ and } \epsilon(G)\geq 0;\\
 <0  &\mbox{if } \Delta(G)<-1  \mbox{ or } \Delta(G)=0,-1 \mbox{ and } \epsilon(G)<0;
\end{cases}
\]
\item[(ii)] \[ G +\cgstar
\begin{cases}
> 0 &\mbox{if } \Delta(G)\ge 1\\
\cgfuzzy 0 &\mbox{if } \Delta(G)=0\\
<0 &\mbox{if } \Delta(G)\le -1\\
\end{cases}
\]
\end{itemize}
\end{thm}

\begin{proof} The proofs are similar to those of Theorems \ref{thm:outcomewithstar} 
and \ref{thm:outcomenostar}, respectively, but without the need of the results in Section 2.
\end{proof}

This paper now has one very nice conclusion. In any collection of star-based number positions (or any sum of Sprigs), outcome is fixed by adding $\cgstar$ and toggling between normal and mis\`{e}re play conventions. This is based on comparing Theorems \ref{thm:outcomewithstar} and \ref{thm:outcomenostar} and Theorem \ref{thm:normal}.

\begin{cor}\label{cor:plusstar}
If $G$ is a collection of star-based numbers, then $o^+(G)=o^-(G+\cgstar)$ and $o^+(G+\cgstar)=o^-(G)$.
\end{cor}

\section{Discussion}

Theorems in Section \ref{sec:miseregeneral} and Corollary \ref{cor:starcolonconjugate} give great hope for further progress in the study of \mpl\ games, especially dicots and ordinal sums. We end with some questions. 

\begin{question}
What can be said about \mpl\ Flowers? 
\end{question}
Generalizing results from $\cgstar$ to $\cgstar_2$ (and other nimbers) is troublesome because $\cgstar_2 + \cgstar_2 \not \equiv 0 \pmod{\mathcal{D}}$. However, the relative simplicity of the partizan component makes this appear tractable.

\begin{question}
What is the largest universe in which $\cgstar \cgc x + \cgstar \cgc \overline{x} \equiv 0$?  
\end{question}
This equivalence does not hold in general and the universe of dicots seem like a natural fit, but it is not clear that there is not a much larger universe where this holds.

\begin{question}
For which positions $G$ is $G \equiv 0\pmod{\mathcal{D}}$?
\end{question}
This question is very broad; it can be difficult to tell from the game tree of a position if it is a sum of positions that we otherwise know to be equivalent to $0$. Less broadly, it would be helpful to have more sums of positions that are equivalent to $0$.

\begin{question}
  Which other classes of positions satisfy the hypothesis of Theorem \ref{thm:G-G}?
\end{question}
Answers to this question give answers to the previous question. We also wonder in particular which star-based positions have this property.

\begin{question}
For which positions $G$ is $o^+(G)=o^-(G+\cgstar)$ and $o^+(G+\cgstar)=o^-(G)$?\end{question}
In addition to Sprigs, the game $\combgame{\{ \overline{1} | \cdot\}}$ has this property. Certainly there are othes.

\bibliography{games3}
\bibliographystyle{plain}
\end{document}